\documentclass[11pt]{amsart}
\usepackage[a4paper,margin=2.4cm]{geometry}
\usepackage{mathtools,amsmath,amsthm,amssymb,amscd,amsfonts}
\usepackage{thm-restate}
\usepackage{hyperref}

\theoremstyle{plain}
\newtheorem{theorem}{Theorem}[section]
\newtheorem{proposition}[theorem]{Proposition}
\newtheorem{lemma}[theorem]{Lemma}

\numberwithin{equation}{section}

\theoremstyle{definition}
\newtheorem{problem}[theorem]{Problem}
\newtheorem{remark}[theorem]{Remark}

\newcommand{\C}{\mathbb{C}}
\newcommand{\R}{\mathbb{R}}
\newcommand{\cF}{\mathcal{F}}
\newcommand{\cZ}{\mathcal{Z}}
\renewcommand\l{\ell}
\renewcommand{\ss}{\text{Ss}}
\DeclareMathOperator{\st}{St}
\DeclareMathOperator{\lk}{Lk}
\DeclareMathOperator{\wed}{Wed}

\DeclareMathOperator{\pic}{Pic}

\title[Stellar subdivisions, wedges, and Buchstaber numbers]{Stellar subdivisions, wedges and Buchstaber numbers}

\author[S. Choi]{Suyoung Choi}
\address{Department of mathematics, Ajou University, 206, World cup-ro, Yeongtong-gu, Suwon 16499,  Republic of Korea}
\email{schoi@ajou.ac.kr}
\author[H. Jang]{Hyeontae Jang}
\address{Department of mathematics, Ajou University, 206, World cup-ro, Yeongtong-gu, Suwon 16499, Republic of Korea}
\email{a24325@ajou.ac.kr}

\date{\today}


\thanks{This work was supported by the National Research Foundation of Korea Grant funded by the Korean Government (NRF-2021R1A6A1A10044950).}
\begin{document}
\begin{abstract}
A seed is a PL~sphere that is not obtainable by a wedge operation from any other PL~sphere. 
In this paper, we study two operations on PL~spheres, known as the stellar subdivision and the wedge, that preserve the maximality of Buchstaber numbers and polytopality.
We construct a new polytopal toric colorable seed from these two operations.
As a corollary, we prove that the toric colorable seed inequality established by Choi and Park is tight.
\end{abstract}
\maketitle

\section{Introduction}
Let $K$ be an $(n-1)$-dimensional simplicial complex on~$[m]=\{1,2,\ldots, m\}$ and $(X,Y)$ a topological pair.
The \emph{polyhedral product}~$(\underline{X}, \underline{Y})^K$ is a subspace of~$X^m$ with respect to~$K$:
$$
	(\underline{X}, \underline{Y})^K \coloneq \bigcup_{\sigma \in K} \left\{ (x_1, \ldots, x_m) \in X^m \mid x_i \in Y \text{ when $i \not\in \sigma$ } \right\}.
$$
One of main considerations is the \emph{moment-angle complex}~$\cZ_K \coloneq (\underline{D^2}, \underline{S^1})^K$ of~$K$, where $D^2 = \{ z \in \C \mid |z|\leq 1 \}$ is the $2$-dimensional disk, and $S^1 = \partial D^2$ is its boundary circle. 
The canonical action of~$S^1$ on~$D^2$ induces the action of the $m$-dimensional torus~$T^m = S^1 \times \cdots \times S^1$ on~$\cZ_K$.
The \emph{Buchstaber number}~$s(K)$ of~$K$ is the maximum number~$r$ such that there is an $r$-dimensional subtorus~$H$ of~$T^m$ freely acting on $\cZ_K$.

It is known that the following inequality holds \cite{Buchstaber-Panov2015, Erokhovets2014}:
$$
	1 \leq s(K) \leq m-n.
$$
The case when $K$ is a PL~sphere with $s(K)=m-n$, particulary when it is a polytopal PL~sphere, is of special interest because many significant toric spaces, including complete non-singular toric varieties, simply called \emph{toric manifolds}, and their topological generalizations such as \emph{quasitoric manifolds} \cite{Davis-Januszkiewicz1991} and \emph{topological toric manifolds} \cite{Ishida-Fukukawa-Masuda2013}, are formed in $\cZ_K/H$, where $K$ is a (polytopal) PL~sphere and $H$ is a free torus action of rank $m-n$ on $\cZ_K$.
Hence, such PL~spheres are said to be \emph{toric colorable} and the characterization of toric colorable PL~spheres is of considerable importance. 

The breakthrough in addressing this problem is the approach by Choi and Park.
The \emph{wedge} is an operation on simplicial complexes that preserves the PL~sphereness and polytopality, and a PL~sphere~$K$ is called a \emph{seed} if it cannot be obtained by a wedge operation from any other PL~sphere.
One advantageous property of the wedge operation in toric topology is that it preserves the maximality of the Buchstaber number \cite{Ewald1986} and does not change the difference of the number of vertices and the dimension.
Therefore, since every toric colorable PL~sphere can be obtained from toric colorable seeds by a sequence of wedges, it is enough to consider toric colorable seeds.
In their paper~\cite{CP_wedge_2}, Choi and Park proved that for a fixed~$m-n$, the number of $(n-1)$-dimensional toric colorable seeds~$K$ with $m$ vertices is finite.
More precisely, $K$ must satisfy the inequality, which we refer to as the \emph{toric colorable seed inequality},
\begin{equation} \label{eqn:inequality}
	m \leq 2^{m-n}-1.
\end{equation} 
This inequality transforms the problem of finding all PL~spheres with a maximal Buchstaber number into a finite problem for each fixed $p \coloneq m-n$.
An interesting question is whether this inequality is tight. 
If the upper bound of~\eqref{eqn:inequality} can be reduced, it would further transform the problem into a smaller finite problem.
It is known that the inequality is tight for~$p=3$ and~$4$.
Denote by~$C^4(7)$ the cyclic polytope of dimension~$4$ with $7$ vertices. 
The boundary complex of~$C^4(7)$ is a seed, and, by~\cite{Erokhovets2011}, its Buchstaber number is maximal.
Hence, the inequality is tight for~$p=3$.
According to~\cite{Choi-Jang-Vallee2023}, there are four PL~spheres that confirms the tightness for~$p=4$ as well.

In this paper, we construct a new polytopal toric colorable seed  from an old one using two operations, the stellar subdivision and the wedge on PL~spheres, that preserve the maximality of Buchstaber numbers and polytopality.
Refer Theorem~\ref{thm: main}.
By using this result, we show that the inequality~\eqref{eqn:inequality} is tight for all $p \geq 3$, even for the class of polytopal PL~spheres.
\begin{restatable*}{corollary}{tight} \label{cor: existence}
    Assume that $p \geq 3$.
    For any $m \leq 2^{p}-1$ and $n \geq 2$ with $p=m-n$, there exists a polytopal $(n-1)$-dimensional seed~$K$ with $m$ vertices such that $s(K) = p$.
\end{restatable*}

From the above corollary, one can precisely determine the conditions on $m$ and $n$ for an $(n-1)$-dimensional seed~$K$ with $m$ vertices to support quasitoric manifolds or topological toric manifolds.
We can ask a stronger question, namely, the conditions on $m$ and $n$ for $K$ to support a toric manifold.
In this case, the inequality~\eqref{eqn:inequality} is not tight.
When $p=3$, $m$ can be at most $6$, which is less than $2^{3}-1=7$, in this class due to~\cite{Kleinschmidt-Sturmfels1991}.
\begin{problem}
    Let $K$ be the underlying complex of a non-singular complete fan with $m$ rays in~$\R^n$.
    Assume that $K$ is a seed. 
    For a fixed $m-n$, find the sharp upper bound of~$m$. 
\end{problem}

\section{Basic operations on simplicial complexes}
Let $K$ be a simplicial complex on $[m]=\{1,2,\ldots,m\}$.
A non-empty element of $K$ is called a \emph{face} of $K$.
Any singleton in $K$ is often called a \emph{vertex}, and a maximal face of $K$ with respect to inclusion is called a \emph{facet}.
We often identify a vertex $\{v\}$ of $K$ with its unique element $v$ as well as with the simplicial complex consisting of the unique vertex $v$.
The \emph{dimension} of a face $\sigma$ is $|\sigma| - 1$, and $K$ is said to be \emph{pure} if the dimensions of all facets are the same.
The \emph{dimension} of pure simplicial complex $K$ is the dimension of its facet.
Throughout this paper, every simplicial complex is assumed to be pure.

For any face $\sigma$ of $K$, the \emph{boundary complex} of $\sigma$ is the simplicial complex $\partial \sigma = \{\tau \in K \mid \tau \subsetneq \sigma \}$, the \emph{star} of $\sigma$ in $K$ is the simplicial complex 
$$
    \st_K(\sigma) = \{\tau \in K \mid \tau \cup \sigma \in K\},
$$
and the \emph{link} of $\sigma$ in $K$ is the simplicial complex 
$$
    \lk_K(\sigma) = \{\tau \in K \mid \tau \cup \sigma \in K,   \sigma \cap \tau = \varnothing  \}.
$$
For another simplicial complex $L$ whose vertex set is disjoint to that of $K$, the \emph{join} of $K$ and~$L$ is the simplicial complex 
$$
    K \ast L = \{\sigma \cup \tau \mid \sigma \in K,   \tau \in L \}.
$$
The \emph{stellar subdivision} of $K$ at $\sigma$ is the simplicial complex 
\begin{equation} \label{eq: ss}
    \ss_\sigma(K) =  K\setminus\st_\sigma(K) \cup v_\sigma \ast \partial \sigma \ast \lk_K(\sigma), 
\end{equation} 
where $v_\sigma$ is a new vertex.

Let $v$ be a vertex of $K$, and $I$ the $1$-dimensional simplicial complex with two vertices $v$ and $v'$.
The \emph{wedge} of $K$ at $v$ is the simplicial complex 
\begin{equation}\label{eq: wed}
    \wed_v(K) = I \ast \lk_K(v) \cup \partial I \ast \{\sigma \in K \mid v \not \in \sigma \}.
\end{equation}
The face $\{v,   v'\}$ of $\wed_v(K)$ is called a \emph{wedged edge} of $\wed_v(K)$.
Note that $\lk_{\wed_{v}(K)}(v')=K$, and $\lk_{\wed_v(K)}(v)$ is isomorphic to $K$ by identifying $v'$ and $v$.
In this sense, $v'$ can be regarded as a copy of $v$.

There is another construction of a wedge using minimal non-faces, known as the \emph{$J$-construction}, established in~\cite{BBCG2015}.
The set of minimal non-faces of $\wed_v(K)$ is obtained by duplicating $v$ as~$v$ and~$v'$ in each minimal non-face of $K$.
In this construction, one can observe that wedge operations are associative and commutative up to suitable vertex identification.
Then we can use the notation~$K(J)$ for a positive integer tuple $J=(j_1, j_2, \dots, j_m)$ to refer the simplicial complex obtained by a sequence of wedge operations with $j_v$ copies $v=v^{(0)}, v'=v^{(1)} \ldots, v^{(j_v-1)}$ of each vertex $v$.
To clarify, we note that $K$ is a subcomplex of $K(J)$, that is, any face of $K$ is a face of $K(J)$.
Moreover, choosing $0 \leq s_v \leq j_v-1$ for each $v \in [m]$, then for a face~$\sigma$ of~$K(J)$, the set~$\sigma^{(s_1, s_2, \dots, s_m)}$ obtained by replacing $v \in \sigma$ by~$v^{(s_v)}$ is also a face of~$K(J)$.
In such a way, we denote the set~$\{1^{(s_1)}, 2^{(s_2)}, \dots, m^{(s_m)}\}$ by~$[m]^{(s_1, s_2, \dots, s_m)}$.
The \emph{assembled face} of $K(J)$ by $(s_1, s_2, \dots, s_m)$ denotes the set $[m]^{(s_1, s_2, \dots, s_m)} \setminus \{v \in [m] \mid j_v = 1 \}$.
The following lemma shows that any assembled face is literally a face of $K(J)$.

\begin{lemma} \label{lem:special_face_of_K(J)}
    Let $K$ be a simplicial complex on $[m]$, and $J=(j_1, j_2, \dots, j_m)$ a positive integer tuple.
    Choose $0 \leq s_v \leq j_v - 1$ for each $v \in [m]$.
    Then $[m]^{(s_1, s_2, \dots, s_m)} \setminus \{v \in [m] \mid j_v = 1 \}$ is a face of $K(J)$.
\end{lemma}
\begin{proof}
    If we prove the case $s_1 = s_2 = \dots = s_m = 0$, then the statement holds by the above argument.
    In this case, it is enough to observe that for a vertex~$v$ and a face~$\sigma$ of~$K$, the set~$\sigma \cup \{v\}$ is a face of~$\wed_v(K)$ by the construction~\eqref{eq: wed}.
\end{proof}

A simplicial complex~$K$ is called a \emph{PL~sphere} if there exists a subdivision of $K$ such that it can be expressed as a subdivision of the boundary complex of a simplex.
A PL~sphere is called a \emph{seed} if it is not a wedge of any other PL~sphere.
Then any PL~sphere $K$ is written as $K = K'(J)$ for some seed $K'$ and some positive integer tuple $J$.

\begin{proposition}[\cite{Ewald-Shephard1974, Choi-Park2016, Choi-Jang-Vallee2023}] \label{prop:ss-wed_preserve_poly}
    If $K$ is a PL~sphere, then both $\ss_\sigma (K)$ and~$K(J)$ are also PL~spheres for any $\sigma \in K$ and any positive integer tuple~$J$.
    In particular, if $K$ is polytopal, then so are $\ss_\sigma (K)$ and~$K(J)$.   
\end{proposition}

For an $(n-1)$-dimensional PL~sphere with $m$ vertices, the \emph{Picard number} of~$K$ is defined as $\pic(K)\coloneq m-n$.
Note that the dimension of $\ss_\sigma (K)$ is $n-1$, and the dimension of $\wed_v(K)$ of $K$ is $n$, while both complexes have $m+1$ vertexes.
Therefore, the stellar subdivision increases the Picard number by $1$, whereas the wedge operation preserves the Picard number:
$$
    \pic(K) = \pic(K(J)) = \pic(\ss_\sigma(K)) -1.
$$
A PL~sphere is said to be \emph{toric colorable} if its Buchstaber number is equal to its Picard number.
\begin{proposition}[\cite{Ewald1986, ewald1996combinatorial}] \label{prop:ss-we_is_colorable}
    Let $K$ be a PL~sphere.
    If $K$ is toric colorable, then both $\ss_\sigma (K)$ and~$K(J)$ are also toric colorable for any $\sigma \in K$ and any positive integer tuple~ $J$:
    $$
        \pic(K) = s(K) = s(K(J)) = s(\ss_\sigma(K)) -1.
    $$
\end{proposition}

The remainder of this section is devoted to the introduction to the suspension operation.
If the vertex $v$ was a ghost vertex, then the equation \eqref{eq: wed} becomes $\partial I \ast K$.
This is called the \emph{suspension} of $K$.
The pair $\{v,   v'\}$ of vertices of $\partial I \ast K$ is called a \emph{suspended pair} of $\partial I \ast K$, and the vertices~$v$ and~$v'$ are called \emph{suspended vertices}.

\begin{proposition} \cite{CP_wedge_2} \label{prop: facets contain two vertices}
  Let $K$ be a PL~sphere.
  For two vertices $v$ and $w$ of $K$, every facet of $K$ contains $v$ or $w$ if and only if $K$ is the wedge with a wedged edge $\{v,   w\}$ or the suspension with a suspended pair $\{v,   w\}$.
\end{proposition}

\begin{lemma} \label{lemma: susp seed}
  Let $K$ be a seed, $J$ a positive integer tuple, and $v$, $w$ two distinct vertices of $K$.
  If every facet of $K(J)$ contains $v^{(s_v)}$ or $w^{(s_w)}$ for some $0 \leq s_v \leq j_v - 1$ and $0 \leq s_w \leq j_w - 1$, then $\{v,   w\}$ is a suspended pair of $K$.
\end{lemma}

\begin{proof}
  Without loss of generality, we may assume that $s_v=s_w=0$.
  For any vertex~$ \neq v, w$ of~$K(J)$, every facet of the link of the vertex still contains~$v=v^{(0)}$ or~$w=w^{(0)}$.
  The repeated link operations with vertices~$u^{(s_u)}$ for all~$u \in [m]$ and~$s_u \geq 1$ gives $K$, and then every facet of~$K$ contains~$v$ or~$w$.
  By Proposition~\ref{prop: facets contain two vertices}, $\{v, w\}$ is a suspended pair of~$K$ since $K$ has no wedged edge.
\end{proof}

The set of minimal non-faces of the join of two simplicial complexes is the union of the sets of minimal non-faces of those.
Then one can easily observe that join and wedge operations are associative and commutative.
Hence we can write a PL~sphere $K$ as \begin{equation} \label{eq: J construction} K = \partial I_1(J_1) \ast \partial I_2(J_2) \ast \dots \ast \partial I_\l(J_\l) \ast L(J_{\l+1}), \end{equation}
where each $I_k$ is a $1$-simplex for $1 \leq k \leq \l$, and $L$ is a seed without a suspended pair.
A PL~sphere is said to be \emph{non-suspended} if it is not the suspension of any other PL~sphere.
The interpretation~\eqref{eq: J construction} of $J$-constructions involving suspensions explains the following.

\begin{proposition} \label{prop: susp wed}
  Let $K$ be a PL~sphere. Then
  \begin{itemize}
    \item $K$ is a suspension if and only if so is its wedge at any vertex $v$ of $K$ not contained in a suspended pair, and
    \item $K$ is a seed if and only if so is the suspension of $K$.
  \end{itemize}
\end{proposition}

\section{Main construction} \label{sec: main}

Let $K$ be a seed on $[m]$ and $J = (j_1, \ldots, j_m)$ a positive integer tuple.
Recall that $K$ is a subcomplex of $K(J)$.
Then the wedged edge~$\{v, v'\}$ is a face of~$\wed_{v}(K)$, and each assembled face of $K(J)$ is a face of $K(J)$ by Lemma~\ref{lem:special_face_of_K(J)}.
Therefore, both $\ss_{\{v, v'\}}(\wed_v(K))$ and $\ss_{\sigma}(K(J))$ are well-defined.

\begin{theorem} \label{thm: main}
  Let $K$ be a seed on~$[m]$ and~$J=(j_1, j_2, \dots, j_m)$ a positive integer tuple such that $j_v \leq 2$ for all $v \in [m]$.
  For any assembled face $\sigma$ of $K(J)$, a new seed is obtained in two ways depending on the cardinality of $\sigma$.
  \begin{enumerate}
    \item If $\sigma = \{v\}$ for a vertex $v$ of $K$, then $\ss_{\{v,   v' \}}(K(J))$ is a suspended seed.
    \item If $|\sigma| > 1$ and $K$ is non-suspended, then $\ss_{\sigma}(K(J))$ is a non-suspended seed.
  \end{enumerate}
\end{theorem}

\begin{proof}
  Without loss of generality, we may assume $\sigma = \{v \mid j_v = 2 \}$.
  First, let $\sigma = \{v\}$.
  Since $\ss_{\{v, v'\}}(\wed_v(K))$ is isomorphic to the suspension of~$K$, it is a seed by Proposition~\ref{prop: susp wed}.

  Next, let $|\sigma| > 1$ and let $K$ be non-suspended.
  Suppose that $\ss_{\sigma}(K(J))$ has two vertices~$x$ and~$y$ such that every facet of~$\ss_{\sigma}(K(J))$ contains $x$ or $y$.
  By Proposition~\ref{prop: facets contain two vertices}, $\{x, y\}$ is a wedged edge or a suspended pair of~$\ss_{\sigma}(K(J))$.

  From the construction \eqref{eq: ss} of stellar subdivisions, one can observe that the facet set~$\cF$ of~$\ss_{\sigma}(K(J))$ is partitioned into two subsets
  $$
    \cF_1=\cF \cap K(J)\setminus \st_{\sigma}(K(J)) \quad \text{ and } \quad \cF_2 = \cF \cap v_{\sigma} \ast \partial \sigma \ast \lk_{K(J)}(\sigma).
  $$
  The subset~$\cF_1$ consists of the facets of~$K(J)$ not containing $\sigma$, and the subset~$\cF_2$ is involved in the facets of~$K(J)$ that contain $\sigma$.
  Let $\cF_\sigma$ be the set consisting of facets of~$K(J)$ containing $\sigma$.
  Then the facet set of~$K(J)$ is $\cF_1 \cup \cF_\sigma$.
  Since every facet in $\cF_2$ contains $x$ or $y$, there are four possibilities: \begin{enumerate}
                                   \item $x=v_{\sigma}$,
                                   \item $x$, $y \in  \sigma$,
                                   \item $x \in  \sigma$ and $y \not \in \sigma$, and
                                   \item $x$, $y \not \in \sigma$.
                                 \end{enumerate}
  We want to prove that all of these four cases lead to contradictions, so there is no such pairs of vertices~$x$ and~$y$.

  \noindent\textsf{\textbf{Case 1}}: $x=v_{\sigma}$.

  There is no facet in $\cF_1$ containing $x$, so every facet in $\cF_1$ contains $y$.
  For any~$z \in \sigma$ with~$z \not= y$, every facet of~$K(J)$ contains~$y$ or~$z$ since the facets in~$\cF_1$ contain~$y$, and the facets in~$\cF_\sigma$ contain~$\sigma \ni z$.
  If $y \in \sigma$, then every facet of~$K(J)$ contains~$y$, but it is impossible that every facet of a PL~sphere shares one vertex.
  Thus, $y$ is not contained in~$\sigma$.
  
  Suppose that $y \neq z'$ for any $z \in \sigma$.
  Then $\lk_{K(J)}(\{z' \mid z \in \sigma \}) = K$ contains $y$ or $z$ for any~$z \in \sigma$ as we discussed in the proof of Lemma~\ref{lemma: susp seed}.
  By Proposition~\ref{prop: susp wed}, $K$ is a suspension or a wedge, but we assumed that $K$ is a non-suspended seed.
  Thus, there exists $y_0 \in \sigma$ such that $y=y_0'$.
  By Lemma~\ref{lemma: susp seed}, $\{y_0, \, z\}$ is a suspended pair of~$K$ for any other vertices~$z \in \sigma$.
  The assumption that $|\sigma| > 1$ ensures the existence of such $z$, and this contradicts to the assumption that $K$ is non-suspended.

  \noindent\textsf{\textbf{Case 2}}: $x$, $y \in \sigma$.

  The face $\sigma$ of $K(J)$ contains $\{x,   y\}$.
  Then $\{x,   y\}$ is a wedged edge of~$K(J)$ since any facet in~$\cF_1$ contains~$x$ or~$y$, and any facet in~$\cF_\sigma$ contains $\{x,   y\}$.
  Note that $\sigma$ is also a face of $K$.
  Then $x$ and $y$ are two distinct vertices of $K$, and then $\{x,   y\}$ is a suspended pair of $K$ by Lemma~\ref{lemma: susp seed}, but $K$ is assumed to have no suspended pair.

  \noindent\textsf{\textbf{Case 3}}: $x \in \sigma$ and $y \not \in \sigma$.

  Note that $\sigma \setminus \{x\}$ is a facet of the boundary complex~$\partial \sigma$.
  Then for any facet~$\tau$ of~$\lk_{K(J)}(\sigma)$, the set~$\{v_\sigma\} \cup \sigma \setminus \{x\} \cup \tau$ is a facet in~$\cF_2$ not containing $x$.
  Hence it have to contain $y$, but there is a facet~$\tau$ of~$\lk_{K(J)}(\sigma)$ not containing $y$ since it is a PL~sphere, which contradicts to the assumption that every facet of~$\ss_\sigma(K(J))$ contains~$x$ or~$y$.

  \noindent\textsf{\textbf{Case 4}}: $x$, $y \not \in \sigma$.

  Every facet in $\lk_{K(J)}(\sigma)$ contains $x$ or~$y$.
  Since it is a PL~sphere, both~$x$ and~$y$ have to appear in~$\lk_{K(J)}(\sigma)$ as its vertices.
  By the assumption that $j_v \leq 2$ for any vertex~$v$ of~$K$ and the definition of~$\sigma$, $\lk_{K(J)}(\sigma)$ is isomorphic to~$K$.
  By Proposition~\ref{prop: facets contain two vertices}, $\lk_{K(J)}(\sigma)$ is a wedge or a suspension, which contradicts to the assumption that $\lk_{K(J)}(\sigma) \cong K$ is a non-suspended seed.
\end{proof}

\tight

\begin{proof}
    We prove by induction on $p\geq3$ that there exists a polytopal toric colorable seed with $m$ vertices for any $p+2 \leq m \leq 2^p-1$, and the one with $m=2^p-1$ is non-suspended.
    Note that the condition~$p+2 \leq m \leq 2^p-1$ is equivalent to~$2 \leq n \leq 2^p-p-1$.

    For $p=3$, the face structures of a pentagon for $n=2$, a crosspolytope for $n=3$, and $C^4(7)$ for $n=4$ are toric colorable seeds.
    These are all polytopal~\cite{Mani1972}, and the one with $n=4$ is non-suspended.    
  
    Assume that the statement holds for some $p \geq 3$.
    Let $K_n$ be an $(n-1)$-dimensional polytopal toric colorable seed of Picard number $p$ for $2 \leq n \leq 2^p -p -1$.
    Now, it is enough to show the existence of a polytopal toric colorable seed of Picard number~$p+1$ for $2 \leq n \leq 2^{p+1}-p-2$.
    
    For $n=2$, note that the face complex of an $(n+p+1)$-gon is a polytopal toric colorable seed of Picard number $p+1$.
    
    For $3 \leq n \leq 2^p-p$, by (1)~of Theorem~\ref{thm: main}, $\ss_{\{v,v'\}} (\wed_v(K_{n-1}))$ for any vertex $v \in K_{n-1}$ is an $(n-1)$-dimensional seed of Picard number $p+1$ that is polytopal by Proposition~\ref{prop:ss-wed_preserve_poly} and toric colorable by Proposition~\ref{prop:ss-we_is_colorable}.

    For $2^p-p+1 \leq n \leq 2^{p+1}-p-2$, we consider $K_{2^p-p-1}$ that is non-suspended by the induction hypothesis.
    Note that $K_{2^p-p-1}$ has $2^p-1$ vertices.
    Let $\sigma = \{1, 2, 3, \dots, k\}$ for $2 \leq k \leq 2^p-1$, and $J$ the positive integer tuple whose first $k$ components are $2$ and the other components are $1$.
    Then the Picard number and the dimension of~$L_k \coloneq \ss_{\sigma}(K_{2^p-p-1}(J))$ are $p+1$ and~$2^p-p+k-2$, respectively.
    For each $2 \leq k\leq 2^p-1$, $L_k$ is indeed an example what we want for $n=2^p-p+k-1$ since $L_k$ is a non-suspended seed by~(2)~of Theorem~\ref{thm: main}, polytopal by Proposition~\ref{prop:ss-wed_preserve_poly}, and toric colorable by Proposition~\ref{prop:ss-we_is_colorable}.
\end{proof}

\begin{remark} 
If we restrict the assumption of the previous corollary to $p \geq 4$, then all the resulting polytopal seeds can be constructed to be non-suspended as follows.
First, one can easily observe that there exists at least one $(n-1)$-dimensional polytopal toric colorable non-suspended seed for each $n=2, \, 3$, and for the dimensions, all PL~spheres are toric colorable.
Then for a Picard number $p$, if there exists an $(n-1)$-dimensional polytopal toric colorable non-suspended seed with $m$ vertices for each $n \geq 2$ and $m \leq 2^{p}-1$ with $p=m-n$, then we can replace the suspension process in the proof of Corollary~\ref{cor: existence} with~(2)~of Theorem~\ref{thm: main} by constructing an $(n+1)$-dimensional non-suspended seed of Picard number $p+1$.
The obstruction is that for $p=3$ and $n=3$, the only toric colorable seed is the boundary of crosspolytope, which is suspended.
However, we can construct an $(n+1)$-dimensional non-suspended seed of Picard number $4$ from the boundary of a pentagon by applying~(2) of Theorem~\ref{thm: main} with $J = (2,2,2,1,1)$.
\end{remark}

\providecommand{\bysame}{\leavevmode\hbox to3em{\hrulefill}\thinspace}
\providecommand{\href}[2]{#2}

\end{document}